\documentclass[a4paper]{amsart}
\usepackage[utf8]{inputenc}
\usepackage[english]{babel}
\usepackage{fancyhdr}
\usepackage{amsmath}
\usepackage{amsfonts,amstext}
\usepackage{amssymb,amsthm,amscd,amsxtra,wasysym,graphicx}
\usepackage{mathrsfs,esint,comment}
\usepackage{xypic,tikz-cd}
\usepackage{enumitem,mathtools,dsfont}
\usepackage{xcolor}

\def\Bibtex{{\rm B\kern-.05em{\sc i\kern-.025em b}\kern-0.08em T\kern-.1667em\lower.7ex\hbox{E}\kern-.125emX}}
 \usepackage{hyperref}
\hypersetup{
     unicode=false,          
    pdftoolbar=true,       
    pdfmenubar=true,       
    pdffitwindow=false,     
    pdfstartview={FitH}, 
    pdftitle={Monge-Amp\`ere equations},   
    pdfauthor={Quang-Tuan Dang},   
    colorlinks=true,  
   linkcolor=purple,         
    citecolor=blue,        
    filecolor=green,      
    urlcolor=blue}

\usepackage{cleveref}
\usepackage{indentfirst}
\newtheorem{theorem}{Theorem}
\newtheorem{lemma}[theorem]{Lemma}
\newtheorem{corollary}[theorem]{Corollary}

\newtheorem{proposition}[theorem]{Proposition}
\newtheorem{definition}[theorem]{Definition}

\theoremstyle{definition}

\numberwithin{theorem}{section}
\numberwithin{equation}{section}
 
\DeclareMathOperator \PSH {{\rm PSH}}

\DeclareMathOperator \MA {{\rm MA}}

\DeclareMathOperator \Vol {{\rm Vol}}

\DeclareMathOperator \Capa {{\rm Cap}}

\def\e{\varepsilon}
\def\f{\varphi}
\def\dc{dd^c}

\begin{document}
	
	\title[Continuity of Monge-Amp\`ere potentials in big cohomology classes]{Continuity of Monge-Amp\`ere potentials in big cohomology classes} 
	\author{Quang-Tuan Dang}

	\address{ Laboratoire de Math\'ematiques D'Orsay, Universit\'e Paris-Saclay, CNRS, 91405 Orsay, France}
	  \email{quang-tuan.dang@universite-paris-saclay.fr}

\address{Institut de Mathematiques de Toulouse,  Universit\'e de Toulouse; CNRS,
118 route de Narbonne, 31400 Toulouse, France}
		\email{quang-tuan.dang@math.univ-toulouse.fr}

	\date{\today}
	\thanks{The author is partially supported by the French ANR project PARAPLUI}

	\begin{abstract} 
	 Extending Di Nezza-Lu's approach \cite{di2017complex} to the  setting 
	 of  big cohomology classes, we prove that solutions of degenerate complex Monge-Ampère
equations on compact K\"ahler manifolds are continuous on a Zariski open set.  
	This   allows us to  show that singular K\"ahler-Einstein metrics on 
	log canonical varieties of general type have continuous  potentials on the ample locus outside of the non-klt part. 
	\end{abstract}
	
	\maketitle
	\tableofcontents

	\section{Introduction}\label{sect: intro}
	
	Finding canonical metrics on complex varieties is a fundamental problem of complex geometry. 
	As evidenced by recent developments in K\"ahler geometry in connection with the Minimal Model Program, it is natural and necessary to allow the varieties $Y$ in question to be singular. Working on a desingularization $\pi: X\rightarrow Y$  is led to consider 
	degenerate complex Monge-Amp\`ere equations of the form
	\begin{equation}\label{eq: bg}
	    (\theta +dd^c \varphi)^n = e^{\lambda \varphi}fdV,
	\end{equation}
	where $\theta$ is a smooth closed real $(1,1)$-form representing a big cohomology class $\alpha$ on $X$, $\lambda \in \{0,\pm 1\}$,  and $f$ is a density of the form $f=e^{\psi_+-\psi_{-}}$, 
	the functions $\psi_+,\psi_{-}$ being quasi-plurisubharmonic on $X$. 
	The integrability properties of $f$ depend on the singularities of $Y$.
	
	 Finding a K\"ahler-Einstein metric on a stable variety $Y$ boils down to solving \eqref{eq: bg} on $X$ for $\lambda=1$ and $f\in L^{1-\delta}(X,dV)$ for all $\delta\in (0,1)$. Building on the variational approach developed in \cite{berman2013variational}, Berman and Guenancia \cite{berman2014kahler} have proved that the equation \eqref{eq: bg} admits a unique finite energy solution $\f$, in the sense of \cite{guedj2007weighted,berman2013variational}. When $\{\theta\}$ is additionally nef,  they established the smoothness of $\varphi$  on a Zariski open set. As in the classical case of Yau \cite{yau1978ricci}, the main difficulty lies in establishing an \emph{a priori $\mathcal{C}^0$-estimate}. Unfortunately the (normalized) solution $\f$ to \eqref{eq: bg} is in general unbounded, so a natural idea is to try and bound  such solution from below 
	 by a reference quasi-plurisubharmonic function. 
	
	This motivates the following general question: for which densities $ f\geq 0$ is the solution $\varphi$  locally bounded 
	  in some Zariski open subset of $X$? The main result of this paper is the following: 

\begin{theorem}
 \label{thm1}
Let $X$ be a compact K\"ahler manifold of dimension $n$ and fix a smooth closed real $(1,1)$-form $\theta$ which represents a big cohomology class. Let $\varphi\in \mathcal{E}(X,\theta)$ be the unique normalized solution  to 
\begin{equation}\label{eq: ma0}
    (\theta+dd^c \varphi)^n = fdV, \; \sup_X \varphi=0. 
\end{equation}
Assume that $f\leq e^{-\phi}$ for some quasi-plurisubharmonic function $\phi$ on $X$. Then
$\f$ is continuous on ${\rm Amp}(\theta)\setminus E_1(\phi)$, where ${\rm Amp}(\theta)$ is the ample locus of $\theta$ and $E_1(\phi)=\{x\in X:\nu(\phi,x)\geq 1 \}$, with $\nu(\phi,x)$ being the Lelong number of $\phi$ at $x$.
\end{theorem}

 Let us recall that $E_1(\phi)$ which is called \emph{the Lelong super-level} set of $\phi$, is an analytic subset of $X$ by Siu's result \cite{siu1974analyticity}.  We refer the reader to Section~\ref{sect: big} for the definition of the \emph{ample locus} of a big cohomology class. 
 
Since $\mu=fdV$ is non-pluripolar, it is known \cite[Section~3]{boucksom2010monge} that there exists
a unique normalized solution $\f\in\mathcal{E}(X,\theta)$, so the point is  to study its regularity.  The idea of the proof is that  we first  use Demailly's equisingular approximation \cite{demailly1992regularization,demailly2015cohomology} (see Theorem~\ref{thm: dem}) to replace $\phi$ by a quasi-psh function $\phi_1$ which has analytic singularities with polar locus $Z$ contained in the set of points where the Lelong number of $\phi$ is greater than or equal to $1$.  
We then adapt the approach of Di Nezza and Lu \cite{di2017complex} (see Theorem~\ref{thm: mainthm}) to prove the continuity of $\f$ in the complement of $E_1(\phi)$ in the ample locus of $\theta$. 
We also prove a slightly stronger version of Theorem \ref{thm1} valid for more singular densities (see Theorem~\ref{mainthm}).

\smallskip

When the density $f$ is smooth in a Zariski open set (i.e. outside an analytic subset),
one expects the solution $\f$ to be smooth in a Zariski open set (see \cite[Question 21, 22]{dinew2016open}), but we are  unable to prove this for the moment. When $f$ belongs to $L^p(X)$ for some $p>1$, the H\"older continuity of $\varphi$  on the ample locus was shown in \cite{demailly2014holder}. The smoothness of $\f$ when $f$ is smooth is largely open.   
 \smallskip
 
Under  the extra assumption that   the class $\{\theta\}$ is  nef, the regularity properties for the solutions for the degenerate Monge-Amp\`ere equation \eqref{eq: ma0} have been studied by many authors (see \cite{boucksom2010monge,berman2014kahler,di2017complex} and the references therein). The strategy in these papers is that one first establishes a relative uniform estimate which allows to adapt classical ideas of Yau \cite{yau1978ricci} and Siu \cite{siu1987lectures} to obtain locally uniform estimates for the Laplacian, and one finally uses Evans-Krylov's general regularity theory to conclude.
In the above case functions in $\mathcal{E}(X,\theta)$ have zero Lelong numbers (see \cite[Corollary 1.8]{guedj2007weighted}, \cite[Theorem 1.1]{darvas2018singularity}). 
Using this property
Di Nezza and Lu \cite{di2017complex} have generalized Kołodziej’s approach \cite{kolodziej1998complex} to establish a relative uniform estimate.
Let us mention that in the general case of a big class even the "least singular" potential $V_{\theta}$ may have positive Lelong numbers.
To overcome this difficulty we exploit fine properties of quasi-plurisubharmonic  envelopes inspired by 
\cite{darvas2020metric,lu2019complex}. 

\smallskip

Our approach allows us to deal with non-nef data. As an application we prove that
 the unique singular K\"ahler-Einstein metric obtained in \cite{berman2014kahler} is continuous on some Zariski open subset.  
 More precisely, we have the following:

\begin{corollary}\label{coro}
Let $(Y,\Delta)$ be a projective log canonical pair of general type, i.e. the canonical line bundle $K_Y+\Delta$ is big.  Then there is a unique  singular K\"ahler-Einstein metric $\omega$ on $Y$ such that 
\begin{align*}
    {\rm Ric}(\omega)=-\omega+[\Delta]
\end{align*} in the weak sense of currents, and such that $\int_Y\omega^n={\rm vol}(K_Y+\Delta)$.
Furthermore $\omega$  has continuous potentials on ${\rm Amp}(K_Y+\Delta)\cap (Y,\Delta)_{\rm reg}\setminus\lfloor\Delta\rfloor$. Here, $(Y,\Delta)_{\rm reg}$ denotes the 
locus of points $p\in Y$ where the pair $(Y,\Delta)$ is log smooth at $p$, i.e.  $Y$ is smooth and $\Delta$ has simple normal crossing (snc) support on a neighborhood of $p$, and $\lfloor \Delta\rfloor$ denotes the integral part of $\Delta$, i.e. if $\Delta=\sum d_i\Delta_i$, then $\lfloor\Delta\rfloor=\sum\lfloor d_i\rfloor\Delta_i$. 
\end{corollary} 

 Recall that the \emph{ample locus} of a big line bundle $L$, denoted by $\textrm{Amp}(L)$ may be defined  as the ample locus of its first Chern class $\alpha=c_1(L)$ (see Definition~\ref{def: amplelocus}).  
\smallskip

 If $K_Y+\Delta$ is additionally \emph{nef}, then the potential of $\omega$  was already known to be smooth on $\textrm{Amp}(K_Y+\Delta)\cap (Y,\Delta)_{\rm reg}\setminus\lfloor\Delta\rfloor$, as follows from the combined arguments of Boucksom-Eyssidieux-Guedj-Zeriahi \cite{boucksom2010monge} (see also \cite{berman2019monge}) and Berman-Guenancia \cite{berman2014kahler}. 
\subsection*{Organization of the paper} The  paper is organized as follows. In Section \ref{sect: pre} we recall basic pluripotential
theory that will be needed later on. The proof of Theorem \ref{thm1} is given in Section \ref{sect: main}, while 
Corollary~\ref{coro} is proved in Section~\ref{sect: kemetric}.
\subsection*{Notations} In the whole article we fix
    \begin{itemize}
        \item $X$ a $n$-dimensional compact K\"ahler manifold,
        \item $dV$ a smooth volume form on $X$,
        \item $\alpha\in H^{1,1}(X,\mathbb{R})$ a big cohomology class, and $\theta$ a smooth  representative of $\alpha$
        \item  a K\"ahler form $\omega$ so that $\omega\geq \theta$.
    \end{itemize}
    
    \subsection*{Acknowledgements} I would like to express my gratitude to my advisors Vincent Guedj and Chinh H. Lu for their help and various interesting discussions.  I warmly thank the referees for useful corrections, comments, and suggestions which improve the presentation of this paper.

	\section{Preliminaries}\label{sect: pre}
   The purpose of this section is to recall some essential materials in  pluripotential theory  which will be used later.
   
\subsection{Quasi-psh functions}\label{sect: quasi-psh}
Recall that an upper semi-continuous function $ \f:X \rightarrow\mathbb{R}\cup\{-\infty\} $
is called {\it quasi-plurisubharmonic} ({\it quasi-psh} for short) if it is locally the sum of a smooth and a plurisubharmonic (psh for short) function. We say that $\f$ is {\it $\theta$-plurisubharmonic}  ({\it $\theta$-psh} for short) if it is quasi-psh, and $\theta+\dc \f\geq 0$ in the sense of currents, where $d^c$ is normalized so that $\dc=\frac{i}{\pi}\partial\Bar{\partial}$. 

By the $\dc$-lemma any closed  positive $(1,1)$-current  $T$ cohomologous to $\theta$ can be written as $T=\theta+\dc\f$ for some $\theta$-psh function $\f$ which is furthermore unique up to an additive constant. 

We let $\PSH(X,\theta)$ denote the set of all $\theta$-psh functions which are not identically $-\infty$. This set is endowed with 
the $L^1(X)$-topology. By Hartog's lemma  $\f\mapsto\sup_X\f$ is continuous in this weak topology. Since the set of  closed positive currents  in a fixed  cohomology class  is compact (in the weak topology), it follows that the set of $\f\in\PSH(X,\theta)$, with $\sup_X\f=0$ is compact. 

Quasi-psh functions are in general singular, and a convenient way to measure their singularities is the Lelong numbers. Let $x_0\in X$. Fixing a holomorphic chart $x_0\in V_{x_0}\subset X$,  the {\em Lelong number} $\nu(\f,x_0)$ of a quasi-psh function $\f$ at  $x_0\in X$ is defined as follows:
\begin{align*}
   \nu(\f,x_0):=\sup\{\gamma\geq 0: \f(z)\leq \gamma\log\|z-x_0\|+O(1), \; \text{on}\; V_{x_0}\}.  
\end{align*} We remark here that this definition does not depend on the choice of local charts.
    In particular, if $\f=\log|f|$ in a neighborhood $V_{x_0}$ of $x_0$, for some holomorphic function $f$, then $\nu(\f,x_0)$ is equal to the vanishing order $\textrm{ord}_{x_0}(f):=\sup\{ k\in\mathbb{N}:D^\gamma f(x_0)=0,\forall\, |\gamma|<k \}$.
 We can also define the \emph{Lelong super-level sets}, for $c>0$, \[E_c(\f):= \{x\in X:\nu(\f,x)\geq c \}.\]  
We also use the notation $E_c(T)$ for a closed positive $(1,1)$-current $T$.
A well-known result of Siu \cite{siu1974analyticity}  asserts that the Lelong super-level sets
$E_c(\f)$ are    analytic subsets of $X$.  We refer the reader to \cite[Remark 3.2]{demailly1992regularization} for a simple proof.

\subsection{Demailly's equisingular approximation}
We next recall   the basic result on the approximation of psh functions by
psh functions with analytic singularities. For details about this, we refer the reader to \cite{demailly1992regularization,demailly2015cohomology}. 

Following Demailly \cite{demailly1992regularization}, a closed positive $(1,1)$-current $T=\theta+\dc\f$ and its global potential $\f$ are said to have \emph{analytic singularities} if there exists $c>0$ such that 
	\begin{align*}
	\f=c\log\left[\sum_{j=1}^{N}|f_j|^2\right]+v,
	\end{align*}
locally on $X$,	where  $v$ is a smooth function and the $f_j$'s are holomorphic functions. 

Thanks to $\dc$-Lemma, the problem of approximating a positive closed $(1,1)$-current is reduced to approximating a quasi-psh function. 
The following result of Demailly \cite{demailly1992regularization,demailly2015cohomology} on the  equisingular approximation  of a quasi-psh function by quasi-psh functions with analytic singularities is crucial:
 
\begin{theorem}[Demailly's equisingular approximation]\label{thm: dem} 
 Let $\f$ be a $\theta$-psh function on $X$. There exists a decreasing sequence of quasi-psh functions $(\f_m)$ such that
\begin{enumerate}
\item $(\f_m)$ converges pointwise and in $L^1(X)$ to $\f$ as $m\to+\infty$, 
\item $\f_m$ has the same
singularities as  $1/2m$ times a logarithm of a sum of squares of holomorphic functions,
    \item $\theta+\dc\f_m\geq -\varepsilon_m\omega$, where $\varepsilon_m>0$ decreases to 0 as $m\to+\infty$,
    \item $\int_Xe^{2m(\f_m-\f)}dV<+\infty$;
    \item $\f_m$ is smooth outside the analytic subset $E_{1/m}(\f)$.
\end{enumerate}
\end{theorem}
\begin{proof}
We refer the reader to \cite[Theorem 1.6, Lemma 1.10]{demailly2015cohomology} for a proof. 	
\end{proof}


\subsection{Big cohomology classes}\label{sect: big}
A cohomology class $\alpha\in H^{1,1}(X,\mathbb{R})$ is  {\it big} 
if it contains a {\it K\"ahler current}, i.e. there is a positive closed current $T\in \alpha$ and $\e>0$ such that $T\geq \e \omega$. Theorem \ref{thm: dem} enables us in particular to approximate a K\"ahler current $T$ inside its cohomology class by K\"ahler currents $T_m$ with analytic singularities, with a very good control of the
singularities. A big class therefore contains plenty of K\"ahler currents with analytic singularities.  

\begin{definition}\label{def: amplelocus}
  We let $\textrm{Amp}(\alpha)$ denote the {\em ample locus} of $\alpha$, i.e. the Zariski open subset of all points $x\in X$ for which there exists a K\"ahler current $T_x\in \alpha$ with analytic singularities such that $T_x$ is smooth in a neighborhood of $x$.  
\end{definition}
  It follows from the work of Boucksom \cite[Theorem 3.17 (ii)]{boucksom2004divisorial} that one can find a single K\"ahler current $T_0\in\alpha$ with analytic singularities such that 
	$$\text{Amp}(\alpha)=X\backslash \textrm{Sing}(T_0).$$ 
In particular $T_0$ is smooth in the ample locus $\textrm{Amp}(\alpha)$.	
 
 Given $\f,\psi\in \PSH(X,\theta)$, we say that $\f$ is {\it less singular} than  $\psi$, and denote by $\psi\preceq\f$, if  there exists a constant $C$ such that $\psi\leq \f+C$ on $X$. We say that $\f,\psi$ have the {\em same singularity type}, and denote by $\f\simeq\psi$ if $\f\preceq\psi$ and $\psi\preceq \f$. 
 \begin{definition}
   A $\theta$-psh function is said to have {\em minimal singularities} if it is less singular  than  any $\theta$-psh function.  
 \end{definition} 
 Such a function is not unique in general, only its class of singularities is. 
Following Demailly, one defines the extremal function
	\begin{align*}
	V_{\theta}:=\sup\{\f\in \PSH(X,\theta): \f\leq 0\}.
	\end{align*}
	 It is a $\theta$-psh function 
	 with minimal singularities. By the analysis above  $V_\theta$ is locally bounded on the ample locus ${\rm Amp}(\alpha)$. Of course
 we have $V_\theta \equiv 0$ if $\theta$ is semi-positive. 
 
 \subsection{Non-pluripolar Monge-Amp\`ere operator}
 Let $\f_1,\cdots,\f_n\in\PSH(X,\theta)$ with minimal singularities. Then they are locally bounded on the ample locus ${\rm Amp}(\alpha)$. Following the construction of Bedford-Taylor \cite{bedford1976dirichlet,bedford1982new} in the local setting, it has been shown in \cite[Section 1.2]{boucksom2010monge} that  the product
 \begin{equation*}
     (\theta+\dc\f_1)\wedge\cdots\wedge(\theta+\dc\f_n)
 \end{equation*}
 is well-defined as a positive Radon measure on ${\rm Amp}(\alpha)$ and it has finite total mass. One can then extend it trivially on the whole $X$. 

In particular, if $\f_1=\cdots=\f_n=\f$ then this procedure defines the (non-pluripolar) Monge-Ampère measure of a function $\f\in\PSH(X,\theta)$ with minimal singularities. For a general $\f\in\PSH(X,\theta)$, its canonical  approximants $\f^j:=\max(\f,V_\theta-j)$, $j>0$ have minimal singularities. 
One can  show that the sequence of Borel positive measures ${\bf 1}_{\{\f>V_\theta-j\}}(\theta+\dc\f^j)$ is increasing in $j$. Its (strong) limit
\begin{equation*}
    \MA_\theta(\f)=(\theta+\dc\f)^n :=\lim_{j\to+\infty}\nearrow {\bf 1}_{\{\varphi>V_{\theta}-j\}}(\theta+\dc\f^j)^n
\end{equation*}
is the \emph{non-pluripolar Monge-Ampère
measure} of $\f$. 
The {\em volume} of a big class $\alpha=\{\theta\}$ is given by the total mass of the non-pluripolar Monge-Ampère
measure of $V_\theta$, i.e.  $${\rm \Vol}(\alpha):=\int_{X}\MA_\theta(V_\theta).$$
 We say that $\f\in\PSH(X,\theta)$ has {\it full Monge-Amp\`ere mass} if $\int_X\MA_\theta(\f)=\Vol(\alpha)$. We let \begin{align*}
     \mathcal{E}(X,\theta):=\left\{\f\in\PSH(X,\theta):\int_X\MA_\theta(\f)=\Vol(\alpha) \right\}
 \end{align*}
 denote the set of $\theta$-psh functions with full  Monge-Amp\`ere mass.
 Note that $\theta$-psh functions with minimal singularities have full  Monge-Amp\`ere mass (see \cite[Theorem 1.16]{boucksom2010monge} for more details), but the converse is not true. 
 
 \smallskip

We  recall here the plurifine locality  of the non-pluripolar product, which will be used several times in this paper.

\begin{lemma}
       Assume that $\f$, $\psi$ are $\theta$-psh function such that $\f=\psi$ on an open set $U$ in the plurifine topology. Then 
       \begin{equation*}
           \mathbf{1}_U\MA_\theta(\f)=\mathbf{1}_U\MA_\theta(\psi).
       \end{equation*}
\end{lemma}

We stress in particular that sets of the form
$\{u < v\}$, where $u$, $v$ are quasi-psh functions, are open in the plurifine topology.
\begin{proof} 
The proof for locally bounded functions can be found in \cite[Corollary 4.3]{bedford1987fine} or \cite[Section 1.2]{boucksom2010monge}. For the general case  we write $\f$ (resp. $\psi$) as the decreasing limits of its canonical approximants $\f^t:=\max(\f,V_\theta-t)$ (resp. $\psi^t:=\max(\psi,V_\theta-t)$). We observe that $\f^t$ (resp. $\psi^t$) is locally bounded on the ample locus ${\rm Amp}(\theta)$.
By the result for locally bounded functions we have
\begin{align*}
  \mathbf{1}_{U\cap (\f>V_{\theta}-t)}\MA_\theta(\f^t)= \mathbf{1}_{U_\cap (\psi>V_{\theta}-t)}\MA_\theta(\psi^t).
\end{align*} Letting $t\to+\infty$, we conclude the proof.
\end{proof}

\subsection{Capacities}
\subsubsection{The Monge-Amp\`ere capacity} 
For the convenience of the reader we recall here a few facts contained in \cite{guedj2017degenerate}.
 Let $K$ be a Borel subset of $X$. The {\it Monge-Amp\`ere capacity} is
\begin{align*}
		\Capa_{\omega}(K):=\sup\left\{\int_K(\omega+\dc u)^n:u\in \PSH(X,\omega), \, -1\leq u\leq 0 \right\}.
		\end{align*}

\begin{lemma}\label{l1} Let $\nu=gdV$ be a Radon positive measure with $0\leq g\in L^p(dV)$ for some $p>1$.
		Then there exists $B>0$ depending on $n$, $p$ $\omega$, $dV$ and $\|g\|_{L^p(dV)}$ such that, for all Borel subsets $K$ of $X$,
		\begin{align*}
		\nu(K)\leq B\cdot\Capa(K)^4.
		\end{align*}
	\end{lemma}
\begin{proof} This result was proved by Kolodziej; see \cite[Section 2.5]{kolodziej1998complex}. We refer the readers to
		 \cite[Proposition 3.1]{eyssidieux2009singular} for an alternative proof.
\end{proof}

	\subsubsection{The generalized capacity}
	We present here a generalization of this notion 
introduced by Di Nezza and Lu \cite{di2017complex, di2015generalized} (see also \cite[Section~ 4.1]{darvas2018monotonicity}).
	\begin{definition}
		Let $\psi\in \PSH(X,\theta)$. We define the \emph{$\psi$-relative capacity} of a Borel subset $K\subset X$ by 
		\begin{align*}
		\Capa_{\theta,\psi}(K):=\sup\left\{\int_K\MA_{\theta}(u):u\in \PSH(X,\theta), \, \psi-1\leq u\leq \psi \right\}.
		\end{align*}
	\end{definition}

Note that when $\theta$ is K\"ahler, a related notion of capacity has been studied in \cite{di2015generalized,di2017complex}. 
The (generalized) Monge-Amp\`ere capacity plays a vital role in establishing uniform estimates for complex Monge-Amp\`ere equation (see e.g. \cite{eyssidieux2009singular,boucksom2010monge,di2015generalized, di2017complex} and the references therein). We shall use the $\psi$-capacity $\Capa_{\theta,\psi}$  in the proof of Theorem~\ref{thm: mainthm}. 

\smallskip

The following results are   important for the sequel.

\begin{lemma}\label{lem: rightcontinuous}
 Fix $\f\in\mathcal{E}(X,\theta)$ and $\psi\in\PSH(X,\theta)$. Then the function \[H(t):=\Capa_{\theta,\psi}(\{\f<\psi-t\}),\quad t\in\mathbb{R},\] is right-continuous and $H(t)\to 0$ as $t\to+\infty$.
\end{lemma}

\begin{proof}
The proof is almost the same as the one of \cite[Lemma 2.6]{di2017complex}
in the K\"ahler case, i.e.  $\theta=\omega$ is K\"ahler. For the reader’s convenience, we give the proof here. The right-continuity is straightforward . For the second statement, we first assume that $\psi\leq V_\theta$. Fix $u\in\PSH(X,\theta)$ such that $\psi-1\leq u\leq \psi$. The generalized comparison principle (\cite[Corollary 2.3]{boucksom2010monge}) yields
\begin{align*}
    \int_{\{\f<\psi-t\}}\MA_\theta(u)\leq\int_{\{\f<u-t+1\}}\MA_\theta(u)&\leq\int_{\{\f<u-t+1\}}\MA_\theta(\f)\\ &\leq\int_{\{\f<V_\theta-t+1\}}\MA_\theta(\f).
\end{align*}
	The last term goes to $0$ as $t\to+\infty$ as $\f\in\mathcal{E}(X,\theta)$. This finishes the proof.  \end{proof} 
\begin{lemma}\label{l27}
	Let $\psi$ be a quasi-psh function such that $\theta+\dc\psi\geq \delta\omega$ for some $\delta\in(0,1)$. Then for any Borel set $K\subset X$,
	\begin{align*}
	    \Capa_\omega(K)\leq \frac{1}{\delta^n} \Capa_{\theta,\psi}(E).
	\end{align*}
\end{lemma}
\begin{proof}
Let $u$ be a $\omega$-psh function such that $-1\leq u\leq 0$. We then have that $\f:={\psi+\delta u}$ is a candidate defining $\Capa_{\theta,\psi}$. It follows that
\begin{align*}
    \delta^n\int_K(\omega+\dc u)^n&\leq \int_K(\theta+\dc\psi+\dc(\delta u))^n\leq \Capa_{\theta,\psi}(K),
\end{align*}
and taking the supremum over all $u$ we get the desired estimate.
\end{proof}

Generalizing \cite[Ineq. (2.3.2)]{kolodziej1998complex}, we have the following result which is a simple consequence of the {\em generalized comparison principle} 
(\cite[Corollary 2.3]{boucksom2010monge}).
	\begin{lemma}\label{p28}
		Let $\psi\in\PSH(X,\theta)$ and $\f\in\mathcal{E}(X,\theta)$. Then for all $t>0$ and $0<s\leq 1$ we have 
		\begin{align*}
		s^n\Capa_{\theta,\psi}(\{\f<\psi-t-s\})\leq\int_{\{\f<\psi-t\}}\MA_{\theta}(\f).
		\end{align*}
	\end{lemma}
	
\begin{proof}
	Let $u$ be a $\theta$-psh function such that $\psi-1\leq u\leq \psi$. We then have
	\begin{align*}
	\{\f<\psi-t-s\}\subset\{\f<su+(1-s)\psi-t \}\subset\{\f<\psi-t \}.
	\end{align*}
	Since $s^n\MA_\theta(u)\leq \MA_{\theta}(su+(1-s)\psi)$ and $\f$ has full Monge-Amp\`ere mass, it follows from  the {generalized comparison principle} (\cite[Corollary 2.3]{boucksom2010monge}) that
	\begin{align*}
	s^n\int_{\{\f<\psi-t-s\}}\MA_{\theta}(u)&\leq\int_{\{\f<\psi-t-s\}}\MA_{\theta}(su+(1-s)\psi)\\
	&\leq \int_{\{\f<su+(1-s)\psi-t\}}\MA_{\theta}(su+(1-s)\psi)\\
	&\leq\int_{\{\f<su+(1-s)\psi-t\}}\MA_{\theta}(\f)\leq \int_{\{\f<\psi-t\}}\MA_{\theta}(\f).
	\end{align*}
Since $u$ was taken arbitrarily as a candidate in the definition of $\Capa_{\theta,\psi}$, the proof  therefore finishes.
\end{proof}

\subsection{Quasi-psh envelopes}\label{sect: envelope}
For a Borel function $h$, we let $P_{\theta}(h)$ denote the largest $\theta$-psh function lying below $h$:
\[ P_\theta(h):=\left(\sup\{\f\in\PSH(X,\theta):\f\leq h \;\text{on}\; X \}\right)^*.\]

	\begin{proposition}
	\label{prop_imp}
	  Fix $\f\in\mathcal{E}(X,\theta)$. Then for any $b>0$, $P_\omega(b\f-bV_\theta)$ is a $\omega$-psh function with full Monge-Amp\`ere mass.    
	\end{proposition}
	The proof given below is inspired by \cite[Lemma 4.3]{darvas2020metric}.
	
\begin{proof} 
 We first  show that the function $P_\omega(b\f-bV_\theta)\not\equiv-\infty$  for all $b>0$.

    For each $j\in\mathbb{N}$ we set $\f_j:=\max(\f,V_\theta-j)$ and $\psi_j:=P_\omega(b\f_j-bV_\theta)$. We observe that $(\psi_j)$ is a decreasing sequence of $\omega$-psh functions, and  $\psi_j\geq -jb$ for each $j$. 
    Therefore the proof would follow if we could show that $\lim_j\psi_j$ is not identically $-\infty$. We let for each $j$, $D_j:=\{\psi_j=b\f_j-bV_\theta\}$ denote the contact set. Observe that the sets $D_j$ are non-empty for $j$ large enough.
     Fix $t>0$. We see that  $$\{\psi_j\leq -t\}\cap D_j=\{\f_j\leq V_\theta-t/b\}\subset\{\f\leq V_\theta-t/b\}.$$ Set $\Tilde{\omega}:=\left( \frac{1}{b}+1\right)\omega$. By Lemma \ref{l2} below and plurifine locality we  have for $j>t/b$, 
    \begin{equation}\label{eq: ineq1}
    \begin{aligned}[c]
        \int_{\{\psi_j\leq -t\}}(\omega+\dc \psi_j)^n &= \int_{\{\psi_j\leq -t\}}{\bf 1}_{D_j}(\omega+\dc \psi_j)^n\\
        &\leq b^n\int_{\{\psi_j\leq -t\}}{\bf 1}_{D_j}(\Tilde{\omega}+\dc\f_j)^n\\
        &\leq b^n\int_{\{\f\leq V_\theta-t/b\}}(\Tilde{\omega}+\dc{\f_j})^n\\
        &=b^n\left(\int_X(\Tilde{\omega}+\dc{\f_j})^n-\int_{\{\f>V_\theta-t/b\}}(\Tilde{\omega}+\dc\f)^n \right),
    \end{aligned}
    \end{equation}
    since $\f_j=\f$ on $\{\f>V_\theta-t/b\}$ for $j>t/b$.
    Suppose by contradiction that $\sup_X\psi_j\rightarrow-\infty$ as $j\to+\infty$. It thus follows that $\{\psi_j\leq -t\}=X$ for $j$ large enough, $t$ being fixed. Hence, for $j>0$ large enough, \eqref{eq: ineq1} becomes
    \begin{equation*}
     \int_X\omega^n\leq b^n\left( \int_X(\Tilde{\omega}+\dc\f_j)^n-\int_{\{\f>V_\theta-t/b\}}(\Tilde{\omega}+\dc\f)^n\right).
    \end{equation*} 
  Letting $j\rightarrow +\infty$, we obtain
  \begin{equation}\label{ineq: contradict}
  \int_X\omega^n\leq b^n\left( \int_X(\Tilde{\omega}+\dc\f)^n-\int_{\{\f>V_\theta-t/b\}}(\Tilde{\omega}+\dc\f)^n\right),
  \end{equation}
  where we have used that $$(\Tilde{\omega}+\dc\f_j)^n=\sum_{k=0}^n\binom{n}{k}(\Tilde{\omega}-\theta)^k\wedge(\theta+\dc\f_j)^{n-k}\to (\Tilde{\omega}+\dc\f)^n$$ in the weak sense of measures on $X$, thanks to \cite[Theorem 2.3, Remark 2.5]{darvas2018monotonicity}. 
    Finally, letting $t\rightarrow+\infty$, in \eqref{ineq: contradict} 
    we obtain a contradiction. Consequently,  $\psi_j$ decreases to a $\omega$-psh function, we infer that $P_\omega(b\f-bV_\theta)$ is a $\omega$-psh function for any $b>0$.
    
  It remains to show that $P_\omega(b\f-bV_\theta)$ has full Monge-Amp\`ere mass. Observe that
   $ P_\omega(b\f-bV_\theta)\geq\frac{b}{A}P_\omega(A\f-AV_\theta)
    $ for $A>b$.
    Using monotonicity of mass (see e.g. \cite[Theorem 1.2]{witt2019monotonicity}), we obtain
    \begin{align*}
        \int_X\left(\omega+\dc P_\omega(b\f-bV_\theta)\right)^n\geq  \left(1-\frac{b}{A}\right)^n\int_X\omega^n +\left(\frac{b}{A}\right)^n\int_X\left(\omega+\dc P_\omega(A\f-AV_\theta)\right)^n. \end{align*}
    Letting $A\to+\infty$ we thus finish the proof. 
\end{proof}	

\begin{lemma}\label{l2}
 Fix $b>0$, $\f$ and $P_\omega(b\f-bV_\theta)\in\PSH(X,\omega)$. 
 Then the measure $(\omega+\dc {P_\omega(b\f-bV_\theta)})^n$ is supported on the contact set $D:=\{P_\omega(b\f-bV_\theta)=b\f-bV_\theta \}$, and $${\bf 1}_D(\omega+\dc{P_\omega(b\f-bV_\theta)})^n\leq  b^n\mathbf{1}_D\left(\left(1+\frac{1}{b}\right)\omega+\dc\f\right)^n. $$
\end{lemma}

\begin{proof}
   We refer the readers to \cite[Lemma 4.4]{darvas2020metric} for a proof of  the  first statement.
    
    For the second one, set $u=\frac{1}{b}P_\omega(b\f-bV_\theta)+V_\theta$. Then $u$ is a $\Tilde{\omega}:=\left( \frac{1}{b}+1\right)\omega$-psh function, and $u\leq \f$. It follows from \cite[Corollary 10.8]{guedj2017degenerate} that
    \begin{align}\label{ineq1: lem2}
        \mathbf{1}_{\{u=\f\}}(\Tilde{\omega}+\dc u)^n\leq \mathbf{1}_{\{u=\f\}}(\Tilde{\omega}+\dc \f)^n.
    \end{align}
    Furthermore, the measure $(\omega+\dc{P_\omega(b\f-bV_\theta)})^n$ is supported on the contact set $D=\{b\f-bV_\theta=P_\omega(b\f-bV_\theta)\}=\{u=\f\}$, hence
    \begin{equation}
       \begin{aligned}[t]
    \label{ineq2: lem2}
        \dfrac{1}{b^n}(\omega+\dc{P_\omega(b\f-bV_\theta)})^n&=\mathbf{1}_{\{u=\f\}}\dfrac{1}{b^n}(\omega+\dc{P_\omega(b\f-bV_\theta)})^n\\
        &\leq\mathbf{1}_{\{u=\f\}}(\Tilde{\omega}+\dc u)^n.
    \end{aligned}
    \end{equation}
   Combining \eqref{ineq1: lem2} and \eqref{ineq2: lem2} we obtain the desired estimate.
\end{proof}

\section{Regularity of solutions} 	

		\subsection{Proof of the {Main Theorem} }\label{sect: main}
		
		In this section we prove Theorem~\ref{thm1}. 
		The key ingredient is an adaptation of Di Nezza-Lu's approach  \cite{di2017complex} (see also \cite{di2015generalized}).
		
		Given a non-negative Radon measure $\mu$ whose total mass is $\Vol(\alpha)$,
 we consider the  Monge-Amp\`ere   equation
 \begin{equation}\label{cmae}
 \MA_{\theta}(\f)=\mu.
 \end{equation}
 The systematic study of such equations in big cohomology classes has been initiated in \cite{boucksom2010monge}. It has been shown there that \eqref{cmae} admits a unique normalized solution $\f\in \mathcal{E}(X,\theta)$ if and only if $\mu$ is a {\it non pluripolar} measure on $X$.  

 Our goal  is to prove the following: 
 
	\begin{theorem}
 \label{mainthm} Assume $\nu=gdV$ to be a Radon measure, with $0\leq g\in L^p(dV)$ for some $p>1$. Let $\mu$ be a non-pluripolar measure such that $\mu(X)=\Vol(\alpha)$.
 Assume that $\mu=fd\nu$, with $f\leq e^{-\phi}$ for some quasi-psh function $\phi$ on $X$. 
 Let $\f\in\mathcal{E}(X,\theta)$ be the unique normalized solution to \eqref{cmae}. Then
 $\f$ is continuous on $\textrm{Amp}(\alpha)\setminus E_{1/q}(\phi)$, where $q$ denotes the conjugate exponent of $p$.
 \end{theorem}
 
Note that Theorem \ref{thm1} in the introduction is a particular case of Theorem~\ref{mainthm}. 
We first  establish this result under an extra assumption. More precisely,  
we have the following theorem, which is closely similar to \cite[Theorem 3.1]{di2017complex} in the case that $\theta$ is K\"ahler.

\begin{theorem}\label{thm: mainthm}
 Let $\f\in\mathcal{E}(X,\theta)$ be normalized by $\sup_X\f=0$. Assume that $\MA_\theta(\f)\leq e^{-\phi}gdV$, for some quasi-psh function $\phi$ on $X$, and $0\leq g\in L^p(dV)$, with $p>1$. 
 Assume that $\phi$ is locally bounded on an open set $U\subset {\rm Amp}(\alpha)$. Then $\f $ is continuous on $U$.
\end{theorem}

\begin{proof}[Proof of Theorem~\ref{thm: mainthm}]	
We fix a $\theta$-psh function $\rho_0$ on $X$  such that  \[\theta+\dc\rho_0\geq 2\delta_0\omega,\]	for some small constant $\delta_0>0$. Replacing $\rho_0$ by $\rho_0-\sup_X\rho_0$, we can always assume that $\rho_0\leq V_\theta$. Moreover, we can choose $\rho_0$ such that it is smooth in the ample locus, with analytic singularities thanks to \cite[Theorem 3.17 (ii)]{boucksom2004divisorial}. 

We will divide the proof in three steps. 
\smallskip

\noindent\textbf{Step 1.} {\it We prove that $\f$ is locally bounded on $U$.} 

We pick $a>0$ so small that $a\phi$ belongs to $\PSH(X,\delta_0\omega)$.
Set $\psi:=\rho_0+a\phi$. We thus have
$\theta+\dc \psi\geq \delta_0\omega$, and $\psi\leq V_\theta+a\phi$.
We claim that
 \begin{equation}\label{eq: lower}
     \f\geq \psi-A,
 \end{equation}
 for  $A>0$ depending only $\delta_0$, $p$, $dV$, $\|g\|_{L^p(dV)}$, and  $\int_Xe^{-2P_\omega(a^{-1}\f-a^{-1}V_\theta)}gdV$. 
\smallskip 
 
We remark that
by Proposition \ref{prop_imp}, for any $b>0$, $P_\omega(b\f-bV_\theta)$ is a $\omega$-psh function with full Monge-Amp\`ere mass,  hence it has  zero Lelong numbers (see \cite[Corollary 1.8]{guedj2007weighted}). Therefore,
  Skoda's integrability theorem \cite{skoda1972sous} ensures that $e^{-P_\omega(b\f-bV_\theta)}$ belongs to $L^q(dV)$ for all $q<+\infty$. In particular, $\int_Xe^{-2P_\omega(b\f-bV_\theta)}gdV$ is finite for any $b>0$.

 The proof of the claim above follows the approach of Di Nezza and Lu \cite{di2017complex,di2015generalized}.
To see this, fix $s\in[0,1]$, $t>0$. Set $d\nu=gdV$, $b=a^{-1}$.
	Using Lemma~\ref{p28} and the assumption on $\MA_\theta(\f)$ we have 
\begin{equation}\label{eq1}
    	\begin{aligned}[c]
	s^n\Capa_{\theta,\psi}(\{\f<\psi-t-s\})&\leq\int_{\{\f<\psi-t\}}\MA_{\theta}(\f)\\
	&\leq \int_{\{\f<\psi-t\}}e^{b(\psi-\f)}e^{-\phi}d\nu \\
	&\leq\int_{\{\f<\psi-t\}} e^{-(b\f-bV_\theta)}d\nu\\
	&\leq \int_{\{\f<\psi-t\}}e^{-P_\omega(b\f-bV_\theta)}  d\nu,
	\end{aligned}
	\end{equation} where  we have used that $b\psi\leq b V_\theta+\phi$ in the third inequality.
Using H\"older's inequality we have
	\begin{align}\label{ine1}
	\int_{\{\f<\psi-t\}}e^{-P_\omega(b\f-b V_\theta)}d\nu\leq \left(\nu({\{\f<\psi-t\}})\right)^{1/2}\left(\int_X e^{-2P_\omega(b\f-bV_\theta)}  d\nu\right)^{1/2}.
	\end{align}
	By Lemma~\ref{l1}, one can find a constant $B>0$ depending on $n$, $p$, $\omega$, $dV$, and $\|g\|_{L^p(dV)}$ such that
	\begin{align*}
	\nu(\cdot)^{1/2}\leq B(\Capa_{\omega}(\cdot))^2.
	\end{align*}
 Since $\theta+\dc\psi\geq\delta_0\omega$ it follows from Lemma~\ref{l27} that $\Capa_\omega\leq \delta_0^{-n}\Capa_{\theta,\psi}$, hence
	\begin{align}\label{ine2}
	\nu(\cdot)^{1/2}\leq B\delta_0^{-2n}\Capa_{\theta,\psi}(\cdot)^2.
	\end{align}
 By \eqref{eq1}, \eqref{ine1} and \eqref{ine2}  we thus get
	\begin{equation}\label{est: cc}
	s^n\Capa_{\theta,\psi}(\{\f<\psi-s-t \})\leq C \Capa_{\theta,\psi}(\{\f<\psi-t \})^2,
	\end{equation} 
	where $C$ depends on $\omega$, $\delta_0$, $n$, $p$, $\|g\|_{L^p(dV)}$, and  $\int_X e^{-2P_\omega(b\f-bV_\theta)}  d\nu$. Set
	\begin{align*}
	H(t):=\left[\Capa_{\theta,\psi}(\{\f<\psi-t \})\right]^{1/n},\quad t>0.
	\end{align*}
	By the estimate \eqref{est: cc} we get
	\begin{align*}
	sH(t+s)\leq C^{1/n}H(t)^2.
	\end{align*} It follows from Lemma~\ref{lem: rightcontinuous} that the function $H$
is right-continuous
 and $H(+\infty)=0$.
We can thus apply \cite[Lemma 2.4]{eyssidieux2009singular} which yields $H(t_0+2)=0$, where  $t_0>0$ is such  that 
\begin{align*}
H(t_0)<\dfrac{1}{2C^{1/n}}.
\end{align*} 
Therefore, for $A=t_0+2$ we have $\f\geq \psi -A$ on $X\backslash P$ for some Borel subset $P$ such that $\Capa_{\theta,\psi}(P)=0$. By Lemma~\ref{l27} we have $\Capa_\omega(P)=0$ so $P$ is a pluripolar set. Hence $\f\geq \psi-A$ everywhere. 

 Using H\"older's inequality it follows from \eqref{eq1} (take $s=1$) that
	 \begin{align*}
	 H(t)^n&\leq\left(\int_Xe^{-2P_\omega(b\f-bV_\theta))}d\nu\right)^{1/2}\left(\int_{\{\f<\psi-t+1\}}d\nu\right)^{1/2}\\
	 &\leq\left(\int_Xe^{-2P_\omega(b\f-bV_\theta)}gdV\right)^{1/2}\left(\dfrac{1}{t-1}\int_{X}|\psi-\f|g dV\right)^{1/2}.
	 \end{align*}
	The last integral is bounded by a uniform constant: using H\"older's inequality again we have $\int_X|\psi-\f|gdV\leq \|g\|_{L^p(dV)}\left(\|\psi\|_{L^{q}(dV)}+\|\f\|_{L^q(dV)}\right)$ with $q=p/(p-1)$.
	Since $\f$  belongs to the compact set of $\theta$-psh functions normalized by $\sup_X \varphi=0$, its $L^{q}$ norm is bounded by  an absolute constant only depending on $\theta,dV$ and $p$. 
	Consequently, we can choose $t_0>0$ to be only dependent on $dV$, $p$, $\|g\|_{L^p(dV)}$,
	and an upper bound for $\int_X e^{-2P_\omega(b\f-bV_\theta)}  d\nu$.

\smallskip

\noindent\textbf{Step 2.} {\it There exists a sequence of functions $\f_j\in\PSH(X,\theta)\cap \mathcal{C}^0(\textrm{Amp}(\alpha))$ which decreases towards $\f$.}

\smallskip
	
	 For convenience, we normalize $\f$ so that $\sup_X\f=-1$. Let $0\geq h_j$ be  a sequence of smooth functions  decreasing to $\f$. Then the sequence of $\theta$-psh functions $\f_j:=P_\theta(h_j)$ decreases to $\f$ as $j\to +\infty$. 
	 Indeed, since the operator $P_\theta$ is monotone, the sequence $\f_j$ is decreasing to a $\theta$-psh function $u$ and since $\f_j\geq \f$ for all $j$ we have $u\geq \f$. Moreover,  $u(x)\leq\f_j(x)\leq h_j(x),\forall\; x\in X$, for all $j$, hence $u(x)\leq \f(x)$,
	 as claimed.
	 Furthermore, we have that $\f_j$ is  continuous  in $\textrm{Amp}(\alpha)$ for each $j$. In fact, \cite[Ineq. (1.2)]{berman2019monge}
	 gives an upper bound on the Monge-Amp\`ere measure of $\f_j$:
	 \begin{equation*}
	     \MA_\theta(\f_j)\leq \mathbf{1}_D\MA_\theta(h_j),\quad D=\{ \f_j=h_j\}.
	 \end{equation*}
	 The equality also holds following the work of Di Nezza and Trapani \cite{di2019monge}. In particular $\MA_\theta(\f_j)$ has an $L^\infty$-density, hence it follows from \cite[Theorem D]{demailly2014holder} that $\f_j$ is H\"older continuous in $\textrm{Amp}(\alpha)$, as claimed.
\smallskip

\noindent\textbf{Step 3.} {\it We show that the decreasing convergence $\f_j\to \f$ is locally uniform on $U$ and finish the proof.}
\smallskip 

	 Fix $\lambda\in (0,1)$. For any $j\in \mathbb{N}$, set
	\begin{align*}
	\psi_j:=\lambda\psi+(1-\lambda)\f_j-(A+2)\lambda,
	\end{align*} where $A>0$ is the uniform constant so that $\f\geq \psi-A$ (see~\eqref{eq: lower}).
	If we pick $\varepsilon_j\leq \frac{\lambda}{2(1-\lambda)}\delta_0$ for $j>0$ big enough, then  $\theta+\dc\psi_j\geq \frac{\lambda}{2}\delta_0 \omega$. We observe by definitions that $\f_j\leq V_\theta$, hence $\psi_j\leq \lambda\psi+(1-\lambda)V_\theta\leq V_\theta+\lambda a\phi$.
	Set
	$$H_j(t):=\left[\Capa_{\theta,\psi_j}(\{\f<\psi_j-t \})\right]^{1/n},\quad t>0.$$ 
	For any $s\in[0,1]$, $t>0$, we can argue as above to obtain
	\begin{align*}
	sH_j(t+s)\leq C^{1/n}H_j(t)^2,
	\end{align*}
  for $C>0$ only depending on $p$, $dV$, $\|g\|_{L^p(dV)}$, $\delta_0$, $\lambda$, and  $\int_Xe^{-2P_\omega (c\varphi-c V_\theta)}gdV$, with $c=(\lambda a)^{-1}$. Let $\chi$ be an increasing convex weight such that $\chi(0)=0$, $\chi(-\infty)=-\infty$, and $\f$ has finite $\chi$-energy (see \cite[Proposition 2.11]{boucksom2010monge}). Since $\f_j\geq \f\geq \psi-A$ we have $\psi_j\leq \f_j-2\lambda$. It follows from Lemma~\ref{p28} (take $s=\lambda$, $\psi=\psi_j+\lambda$)  that
	\begin{align*}
	\lambda^n\Capa_{\theta,\psi_j}\{\f<\psi_j\}&\leq\int_{\{\f<\psi_j+\lambda\}}\MA_{\theta}(\f)\leq\int_{\{\f<\f_j-\lambda  \}} \MA_{\theta}(\f)\\
	&\leq \dfrac{1}{-\chi(-\lambda)}\int_X(-\chi\circ(\f-\f_j))\MA_\theta(\f).
	\end{align*}
	The latter converges to $0$ as $j\rightarrow +\infty$ since $\f_j$ decreases to $\f$, namely $H_j(0)$ goes to $0$ as $j\to+\infty$. We thus take $j>0$ so big  that $H_j(0)\leq 1/(2C^{1/n})$. It thus follows from  \cite[Remark 2.5]{eyssidieux2009singular} that $H_j(t)=0$ if $t\geq t_0$ where $t_0\leq 4C^{1/n}H_j(0)$. We then have
	\begin{align}\label{ineq: thm11}
	\f\geq \lambda\psi+(1-\lambda)\f_j-(A+2)\lambda-4C^{1/n}H_j(0).
	\end{align} We have $H_j(0)\to 0$ as $j\to\infty$.
 Letting $j\rightarrow+\infty$, we thus obtain
	\begin{align}\label{eq: last}
	\lim_{j\rightarrow+\infty}\inf_K(\f-\f_j)\geq -\lambda(\sup_K|\psi|+A+2),
	\end{align} for any compact subset $K$ of $U$.  We observe by definition that $\psi=\rho_0+a\phi$ is bounded on $K$.
	Finally, letting $\lambda\rightarrow 0$ in~\eqref{eq: last} we obtain that the convergence $\f_j\to\f$ is locally uniform on $U$, hence $\f$ is continuous on $U$.
\end{proof}

\begin{proof}[Proof of Theorem \ref{mainthm}] 
By rescaling, we may  assume without loss of generality that $\phi$ is a $\omega$-psh function.
By Theorem~\ref{thm: dem}, we can find a quasi-psh function ${\phi}_m$ with analytic singularities such that $\omega+\dc{\phi}_m\geq -\varepsilon_m\omega$ for  $\varepsilon_m>0$ decreasing to $0$ as $m\to +\infty$, and $$\int_{X}e^{2m({\phi}_m-{\phi})}dV<+\infty.$$
Note that ${\phi}_m$ is smooth outside the analytic set $E_{1/m}(\phi)\subset X$. We see that
\begin{align*}
    \MA_\theta(\f)\leq e^{-\phi_m} e^{({\phi}_m-{\phi})} gdV.
\end{align*} Set now
	 $g'=e^{({\phi}_m-{\phi})}g$. We choose $m=[q]$, where $[q]$ denotes the integer part of $q$. We have
	  $2m>q$, hence there is a constant $p'>1$ such that $\frac{1}{p'}=\frac{1}{p}+\frac{1}{2m}$.  It follows from  H\"older's inequality that $g'\in L^{p'}(dV)$. 
	 Observe that $\phi_m$ is smooth in the complement $X\setminus E_{1/m}(\phi)$ of the analytic set $E_{1/m}(\phi)\subset X$, in particular it is locally bounded on $\textrm{Amp}(\alpha)\setminus E_{1/m}(\phi)\supset \textrm{Amp}(\alpha)\setminus E_{1/q}(\phi)$.  We can thus apply Theorem~\ref{thm: mainthm} to complete the proof. 
	 
	 In particular if $p=\infty$ then we can choose $m=1$ to conclude the proof of Theorem~\ref{thm1}.
\end{proof}	 

\subsection{K\"ahler-Einstein metrics on log canonical pairs of general type}\label{sect: kemetric}		

\subsubsection{Log canonical singularities}

A \emph{pair} $(Y,\Delta)$ is by definition a connected  complex normal projective variety $Y$ and an effective Weil $\mathbb{Q}$-divisor $\Delta$. We will say that the pair $(Y,\Delta)$ has \emph{log canonical singularities} if $K_Y+\Delta$ is $\mathbb{Q}$-Cartier, and if for some (or equivalently any) log resolution $\pi:X\to Y$ of $(Y,\Delta)$, we have
\begin{equation*}
    K_X=\pi^*(K_Y+\Delta)+\sum_ia_iE_i,
\end{equation*}
where $E_i$ are either exceptional divisors or components of the strict transform of $\Delta$, and the coefficients $a_i\in\mathbb{Q}$ satisfy the inequality $a_i\geq -1$ for all $j$. The divisor $\sum_i E_i$ has simple normal crossing support. We denote the singular set of $Y$ by $Y_{\rm sing}$ and let $Y_{\rm reg}:=Y\setminus Y_{\rm sing}$.

Let $m$ be a positive integer such that $m(K_Y+\Delta)$ is Cartier. If we choose $\sigma$ a local generator of $m(K_Y+\Delta)$ defined on an open subset $U$ of $Y$, then $
    (i^{mn^2}\sigma\wedge\Bar{\sigma})^{1/m}
$ defines a smooth volume form on $U\cap (Y_{\rm reg}\setminus \text{Supp}(\Delta))$. If $
f_i$ is a local equation of $E_i$ around a point $\pi^{-1}(U)$, then we can see that 
\begin{equation}\label{eq: adapted}
\pi^*\left( i^{mn^2}\sigma\wedge\Bar{\sigma} \right)^{1/m}=\prod_i|f_i|^{2a_i}dV
\end{equation} locally on $\pi^{-1}(U)$ for some local volume form $dV$.

The previous construction leads to the following \emph{adapted measure} which is introduced in \cite[Sect. 6]{eyssidieux2009singular}:
\begin{definition}
    Let $(Y,\Delta)$ be a pair and let $h$ be a smooth hermitian metric on the $\mathbb{Q}$-line bundle $\mathcal{O}_Y(K_Y+\Delta)$. The corresponding {\em adapted measure} $\mu_{Y,h}$ on $Y_{\rm reg}$ is locally defined by choosing a nowhere zero  section $\sigma$ of $\mathcal{O}_Y(m(K_Y+\Delta))$ over a small open set $U$ and setting
    \begin{equation*}
        \mu_{Y,h}:=\frac{(i^{mn^2}\sigma\wedge\Bar{\sigma})^{1/m}}{|\sigma|_{h^m}^{2/m}}.
    \end{equation*}
\end{definition}  The point of the definition is that the measure $\mu_{Y,h}$ does not depend on the choice of $\sigma$. This measure can be extended by zero across $Y_{\rm sing}\cup \textrm{Supp}(\Delta)$. 
Remark that the restriction $\sigma|_{Y_{\rm reg}}$ can be viewed as a meromorphic form with a pole of order $md_i$ on $\Delta_i$ where $\Delta=\sum_i d_i\Delta_i$ is the decomposition of $\Delta$ into prime divisors (see \cite[Sect. 6.3]{eyssidieux2009singular}). The Lelong-Poincar\'e formula yields \begin{equation*}
-\dc\log\mu_{Y,h}=[\Delta]-i\Theta_h(K_Y+\Delta)
\end{equation*} on $Y_{reg}$, where $[\Delta]$ is the integration current on $\Delta$.

\subsubsection{K\"ahler-Einstein metrics} 
Let $(Y,\Delta)$ be a log canonical pair, with $n=\dim_\mathbb{C}Y$. Assume that $K_Y+\Delta$ is a big $\mathbb{Q}$-line bundle. 

 We next recall the notion of (negatively curved) K\"ahler-Einstein metric attached to a pair $(Y,\Delta)$. There are various equivalent definitions for such an object (e.g. in \cite{eyssidieux2009singular,boucksom2010monge,berman2014kahler}), we choose here the following definition in the sense of \cite[Sect. 3]{berman2014kahler}.

\begin{definition} We say that a closed positive current $\omega_{KE}\in c_1(K_Y+\Delta)$ on $Y$ is 
a \emph{(singular) K\"ahler-Einstein} metric (\emph{KE} for short) with negative curvature for  $(Y,\Delta)$ if
\begin{enumerate}
	\item The non-pluripolar product $\omega_{KE}^n$ defines a (locally) absolutely continuous measure on $Y_{\rm reg}$ with respect to $dz\wedge d\bar{z}$ and $\log(\omega_{KE}^n/dz\wedge d\bar{z})\in L^1_{\rm loc}(Y_{\rm reg})$, 
	where $z=(z_i)$ are local holomorphic coordinates;
	\item ${\rm Ric}(\omega_{KE})=-\omega_{KE}+[\Delta]$ on $Y_{\rm reg}$;
	\item $\int_{Y_{\rm reg}} \omega^n_{KE}={\rm vol}(K_Y+\Delta)$.
\end{enumerate} 
\end{definition}
The condition (1) allows us to define (on $Y_{\rm reg}$) the Ricci curvature of $\omega_{KE}$ by setting $\textrm{Ric}(\omega_{KE}):=-\log(\omega_{KE}^n)$. 
Another way of thinking of this is to interpret the positive measure $\omega_{KE}^n|_{Y_{\rm reg}}$  as a singular metric on $-K_{Y_{\rm reg}}$ whose curvature is then $\textrm{Ric}( \omega_{KE})$ by definition.

Let $h$ be a smooth hermitian metric on $K_Y+\Delta$ with curvature $\eta$.
 Finding a singular 
K\"ahler-Einstein metric is equivalent to solving the following Monge-Amp\`ere equation for an $\eta$-psh $\phi$ with full Monge-Amp\`ere mass
\begin{equation}\label{eq: cmae}
(\eta+\dc\phi)^n=e^{\phi+c} \mu_{Y,h
},
\end{equation} for some $c\in\mathbb{R}$. Indeed, we set $\omega:=\eta+\dc\phi$. Since $\phi$ is locally integrable $\omega$ satisfies the condition (1) (see Eq.~\eqref{eq: adapted}).
 We also have
\begin{equation*}
    {\rm Ric}(\omega)=-\dc\log\omega^n=-\dc\phi-\dc\log\mu_{Y,h}=-\dc\phi-\eta+[\Delta]
\end{equation*} on $Y_{\rm reg}$. Condition (3) is clearly satisfied.

\smallskip

We now prove Corollary~\ref{coro} in the introduction. Assume the initial  pair $(Y,\Delta)$ is lc of general type, i.e. the canonical bundle $K_Y+\Delta$ is big.  We consider a log resolution $\pi:(X,D)\rightarrow(Y,\Delta)$ of the pair, with $K_X+D=\pi^*(K_Y+\Delta)$. Here, $D=\sum a_i D_i$ is a $\mathbb{R}$-divisor with simple normal crossing support (snc for short) on $X$, consisting of $\pi$-exceptional divisors with coefficients in $(-\infty,1]$, and of the strict transforms of the components of $\Delta$ with  coefficients in $(0,1]$. The (singular) K\"ahler-Einstein metric  $\omega_{KE}$ for $(X,D)$, or equivalently the pull-back of the (singular) KE metric for $(Y,\Delta)$ by $\pi$ 
can be written as $\omega_{KE}=\theta+\dc \f$ where $\theta=\pi^*\eta$
is a smooth representative of $c_1(K_X+D)$ and $\f$ is a $\theta$-psh function (with full Monge-Amp\`ere mass)
solving the following Monge–Amp\`ere equation
\begin{equation}\label{eq: ofbg}
    \MA_\theta(\f)=\frac{e^\f dV}{\prod_i|s_i|^{2a_i}}.
\end{equation} Here $s_i$ are sections cutting out  $D_i$ above, $|\cdot|_i$ are smooth hermitian metrics on associated line bundles
$\mathcal{O}_X(D_i)$, and $dV$ is the smooth
volume form with prescribed Ricci curvature. We let $D_{\rm nklt}$ denote the non-klt part of $D$, i.e. $D_{\rm nklt}:=\cup_{a_i=1}D_i$. We remark that $\f$ goes to $-\infty$ near $D_{\rm nklt}$ as $|s|^{-2}$ is not integrable.

As a consequence of Theorem \ref{thm: mainthm} we have the following

\begin{corollary}\label{cor1}
There is a unique solution $\f\in\mathcal{E}(X,\theta)$ to the equation~\eqref{eq: ofbg} which is continuous on $\textrm{Amp}(\theta)\setminus D_{\rm nklt}$.
\end{corollary}

\begin{proof} The existence  of a unique solution $\f\in\mathcal{E}(X,\theta)$ to the Monge-Amp\`ere equation \eqref{eq: ofbg} follows from \cite[Theorem 4.2]{berman2014kahler}. It remains to prove the continuity of $\f$. 
 It is convenient to 
 differentiate the "klt part" of $D$ from its "non-klt part",  
 so we set $g=\prod_{a_i<1}|s_i|^{-2a_i}\in L^p$ for some $p>1$, and $\phi=\sum_{a_i=1}2\log|s_i|$. We see that $\phi$ is smooth outside of the non-klt locus $D_{\rm nklt}$, in particular it is locally bounded on $\textrm{Amp}(\theta)\setminus D_{\rm nklt}$.
 Since $\f$ is bounded from above, we can therefore apply Theorem~\ref{thm: mainthm} to complete the proof.
\end{proof}

With the notations above, since $\{\theta\}$ is a pull-back by $\pi$ of a big class,
 we have
 $$\textrm{Amp}(\theta)=\pi^{-1}(\textrm{Amp}(K_Y+\Delta))\setminus \textrm{Exc}(\pi)$$ which projects onto $\textrm{Amp}(K_Y+\Delta)\cap (Y,\Delta)_{\rm reg}$.
We next  observe that the projection of  the non-klt locus $D_{\rm nklt}$ is contained in $(Y,\Delta)_{\rm sing}\cup \lfloor\Delta \rfloor$, with $\lfloor\Delta \rfloor=\sum_{d_i=1}\Delta_i$. The proof of Corollary~\ref{coro} thus follows.

	\bibliographystyle{plain}
	\bibliography{bibfile}	
	
\end{document}